\documentclass[11pt]{amsart}
\usepackage{latexsym,amssymb,url,todonotes, a4wide,verbatim,enumerate,array}

\usepackage{hyperref}
\hypersetup{colorlinks=true, citecolor=darkblue, linkcolor=darkblue}
\newcolumntype{C}[1]{>{\centering\arraybackslash}p{#1}}

\usepackage{tikz}
\usetikzlibrary{calc,through,backgrounds,shapes,matrix}
\newcounter{intege}
\newcommand{\polygon}[5]{ 
  \foreach \t in {1,...,#3} {
    \coordinate (#2\t) at ($#1+(90-\t*360/#3:#4)$);
  }
  \draw[thin,black,fill=white,opacity=0.3,densely dashed] #1 circle (#4);
  \setcounter{intege}{1}
  \pgfmathsetcounter{intege}{1}
  \foreach \object in {#5}{
    \ifthenelse{\not\equal{\object}{}}{
      \filldraw[black] ($#1+($(90-\theintege*360/#3:#4)$)$) circle(2pt);
    }{
      \filldraw[black] ($#1+($(90-\theintege*360/#3:#4)$)$) circle(2pt);
    }
    \node[inner sep=0pt] at ($#1+($1.12*(90-\theintege*360/#3:#4)$)$) {$\scriptstyle\object$};
    \pgfmathsetcounter{intege}{\theintege+1}
    \setcounter{intege}{\theintege}
  }
}

\definecolor{darkblue}{rgb}{0,0,0.7} 
\newcommand{\darkblue}{\color{darkblue}} 
\newcommand{\defn}[1]{\emph{\darkblue #1}} 

\theoremstyle{plain}
   \newtheorem{theorem}{Theorem}[section]
   \newtheorem{proposition}[theorem]{Proposition}
   \newtheorem{lemma}[theorem]{Lemma}
   \newtheorem{corollary}[theorem]{Corollary}

   \newtheorem{prop-def}[theorem]{Proposition-Definition}
\theoremstyle{definition}
   \newtheorem{definition}[theorem]{Definition}
   
   \newtheorem{example}[theorem]{Example}
   
   \newtheorem{remark}[theorem]{Remark}

\numberwithin{equation}{section}

\newcommand{\NC}[1]{\operatorname{NC}(W,#1)} 
\newcommand\CC{{\mathbb{C}}}
\newcommand\NN{{\mathbb{N}}}
\newcommand\RR{{\mathbb{R}}}
\newcommand\QQ{{\mathbb{Q}}}

\newcommand\Sym{{\operatorname{Sym}}}
\newcommand\Aut{{\operatorname{Aut}}}
\newcommand\Cox{{\operatorname{\mathcal{C}}}} 
\newcommand\tr{{\operatorname{tr}}}
\newcommand\GL{{\operatorname{GL}}}
\newcommand\Gal{{\operatorname{Gal}}} 

\newcommand\Sp{{\operatorname{Sp}}}

\newcommand{\one}{{1\!\!1}} 
\newcommand{\charpol}{{\chi}}
\begin{document}

\title[On non-conjugate Coxeter elements]{On non-conjugate Coxeter elements in\\ well-generated reflection groups}

\author[V.~Reiner]{Victor Reiner$^*$}
\address[V.~Reiner]{School of Mathematics\\University of Minnesota\\Minneapolis, MN 55455, USA}
\email{reiner@math.umn.edu}
\thanks{$^*$Supported by NSF grant DMS-1001933.}

\author[V.~Ripoll]{Vivien Ripoll$^\dagger$}
\address[V.~Ripoll]{Fakult\"at f\"ur Mathematik, Universit\"at Wien, Oskar-Morgenstern-Platz 1,
1090 Wien, Austria}
\email{vivien.ripoll@univie.ac.at}
\thanks{$^\dagger$Supported by the Austrian Science Foundation FWF, grants Z130-N13 and F50-N15, the latter in the framework of the Special Research Program ``Algorithmic and Enumerative Combinatorics''.}

\author[C.~Stump]{Christian Stump$^\ddagger$}
\address[C.~Stump]{Institut f\"ur Mathematik, Freie Universit\"at Berlin, Germany}
\email{christian.stump@fu-berlin.de}
\thanks{$^\ddagger$Supported by the German Research Foundation DFG, grant STU 563/2-1 ``Coxeter-Catalan combinatorics".}

\date{\today}

\begin{abstract}  
  Given an irreducible well-generated complex reflection group~$W$ with Coxeter number~$h$, we call a Coxeter element any regular element (in the sense of Springer) of order~$h$ in~$W$; this is a slight extension of the most common notion of Coxeter element.
  We show that the class of these Coxeter elements forms a single orbit in~$W$ under the action of reflection automorphisms.
  For Coxeter and Shephard groups, this implies that an element~$c$ is a Coxeter element if and only if there exists a simple system~$S$ of reflections such that~$c$ is the product of the generators in~$S$.
  We moreover deduce multiple further implications of this property.
  In particular, we obtain that all noncrossing partition lattices of~$W$ associated to different Coxeter elements are isomorphic.
  We also prove that there is a simply transitive action of the Galois group of the field of definition of~$W$ on the set of conjugacy classes of Coxeter elements.
  Finally, we extend several of these properties to Springer's regular elements of arbitrary order.
\end{abstract}

\maketitle

\setcounter{tocdepth}{1}
\tableofcontents

\section{Background and main results}
\label{sec:intro}

Let $V=\CC^n$, and consider a finite subgroup~$W$ of $\GL(V) \cong \GL_n(\CC)$.
One calls~$W$ a \defn{complex reflection group} if it is generated by its subset~$R$ of \defn{reflections}, that is, the elements~$r \in W$ for which the fixed space $\ker(r-\one) \subseteq V$ is a hyperplane.
Results of G.~C.~Shephard and J.~A.~Todd~\cite{ShephardTodd} and of C.~Chevalley~\cite{Chevalley} distinguish complex reflection groups as those finite subgroups of $\GL_n(\CC)$ for which the invariant subalgebra of the action on $\Sym(V^*)\cong\CC[x_1,\ldots,x_n] $ yields again a polynomial algebra, $\Sym(V^*)^W=\CC[f_1,\ldots,f_n]$.
While the basic invariants $f_1,\ldots,f_n$ are not unique, they can be chosen homogeneous, and then their degrees $d_1 \leq \cdots \leq d_n$ are uniquely determined and called the \defn{degrees} of~$W$.
The group~$W$ is called \defn{irreducible} if it does not preserve a proper subspace of~$V$.
An important subclass of irreducible complex reflection groups are those that are \defn{well-generated}, that is, for which there exists a subset of~$n$ reflections that generate~$W$.
In particular, this subclass contains all (complexifications of) irreducible \defn{real} reflection groups inside $\GL_n(\RR)$, and as well the subclass known as \defn{Shephard groups}, described in detail in Section~\ref{ssec:shephard}.

\bigskip

The results in this article mainly concern Coxeter elements in irreducible well-generated reflection groups, and are introduced below in Sections~\ref{ssec:def-cox},~\ref{ssec:charac-coxeter} and~\ref{ssec:galois-conj}.
The more general case of Springer's regular elements is then presented in Section~\ref{ssec:regular}.

\medskip

\subsection{Coxeter elements and noncrossing partition lattice}
\label{ssec:def-cox}
For an irreducible real reflection group~$W$, let~$\mathfrak{C}$ be a chamber of the arrangement of reflecting hyperplanes in~$\RR^n$.
To this chamber, one can associate a distinguished set~$S = \{s_1,\ldots,s_n\}\subseteq R$ of \defn{Coxeter generators} for~$W$ obtained by taking those reflections defined by the boundary hyperplanes of~$\mathfrak{C}$, see~\cite{Humphreys}.
The pair $(W,S)$ is a \defn{Coxeter system}, and every finite Coxeter system can be obtained this way from a finite real reflection group.
A \defn{Coxeter element} in~$W$ is classically defined as the product of the reflections in~$S$ in any order, see H.~S.~M.~Coxeter in \cite{Cox1951}.
It thus depends on the choice of the chamber~$\mathfrak{C}$ and on the order of the factors.
It is well known that, however, the set of such Coxeter elements forms a single conjugacy class in~$W$, and that the order of a Coxeter element is equal to the highest degree~$d_n$ (for basic results on Coxeter elements, we refer to~\cite[Ch.~3.16--3.19]{Humphreys} or \cite[Ch.~29]{Kan2001}).
Coxeter elements play an important role in the theory of (finite) Coxeter groups.
In particular, they are crucial in Coxeter-Catalan combinatorics, namely in the context of noncrossing partitions, cluster complexes, generalized associahedra, Cambrian fans and lattices, and subword complexes.
For details about these concepts, we refer to~\cite{Arm2006,Rea2007,PS2011} and the references therein.
Work of T.~Brady and C.~Watt~\cite{BradyWatt} and of D.~Bessis~\cite{Bessis1} on the braid group of~$W$ show the importance of the \defn{$W$-noncrossing partition lattice}~$\NC{c}$ associated to a Coxeter element~$c \in W$.
This poset $\NC{c}$ is defined as the principal order ideal generated by any Coxeter element~$c$, that is,
\begin{equation}
  \NC{c} = [\one,c]_W=\{w \in W \,|\, \one \leq_R w \leq_R c\}.
  \label{eq:ncp}
\end{equation}
Here, $\leq_R$ denotes the \defn{absolute order} on~$W$ given by
\begin{equation}
  x \leq_R y \iff \ell_R(x) + \ell_R(x^{-1}y) =\ell_R(y)
  \label{eq:abs}
\end{equation}
where~$\ell_R(x)$ is the \defn{absolute length} of~$x$ with respect to the set~$R$ of all reflections in~$W$, and~${\one \in W}$ denotes the identity element.

Further work of D.~Bessis~\cite{Bessis2} shows how to generalize the noncrossing partition lattice to all irreducible well-generated groups.
The definitions of absolute length and absolute order in~\eqref{eq:abs} still make sense when the group is not real, but one needs a replacement for the notion of Coxeter elements. 
This is provided by the following notion of regularity from T.~A.~Springer~\cite{Spr1974}.
An element~${w \in W}$ is called \defn{regular} if~$w$ has an eigenvector~$v$ in the complement of the reflecting hyperplanes, so that~$W$ acts freely on the orbit of~$v$.
Say that~$w$ is~$\zeta$-regular if $w(v)=\zeta v$ in this situation.
The multiplicative order~$d$ of~$w$ within~$W$ will be the same as that of~$\zeta$ within $\CC^\times$, and one calls~$d$ a \defn{regular number} for~$W$.
A simple characterization of regular numbers was obtained by G.~I.~Lehrer and T.~A.~Springer~\cite{LS1999}, and later proven uniformly by G.~I.~Lehrer and J.~Michel~\cite{LehrerMichel}.
This characterization implies that for irreducible well-generated groups, the \defn{Coxeter number} $h=d_n$ is always a regular number, and that it is the highest regular number possible.
It turns out that for a real reflection group, the class of usual Coxeter elements corresponds to the class of $e^{2i\pi/h}$-regular elements.
D.~Bessis thus replaced the Coxeter element~$c$ in a real reflection group with an $e^{2i\pi/h}$-regular element in an irreducible well-generated group, see~\cite[Definition~7.1]{Bessis2}.
In the present paper, we consider the following more general definition.

\begin{definition}
  \label{def:cox}
  Let~$W$ be an irreducible well-generated complex reflection group~$W$. 
  A \defn{Coxeter element} in~$W$ is a regular element in~$W$ of order $h=d_n$.
\end{definition}

As mentioned above, the usual definition is more restrictive than the one given here: a Coxeter element is classically taken to be regular for the specific eigenvalue $e^{2i\pi/h}$, and this notion is, for real reflection groups, equivalent to the definition using the product of the reflections through the walls of a chamber in the reflection arrangement. 
Both definitions (the classical one, and the extended one from Definition~\ref{def:cox}) have been used in the literature, but their subtle distinction has been sometimes source of confusion. 
For example, the statement of Theorem~C in~\cite[Ch.~32.2]{Kan2001} is erroneous.
Also, some results of~\cite{BessisR} used the extended definition, although they relied on results from~\cite{Bessis2} which used only the restrictive definition.
One of the purposes of this article is to clarify these confusions.
We will next see that the more general definition provides new insights even for real reflection groups.

\begin{example}
\label{ex:H2}
(cf.~\cite[Example 2.7]{Mil2014})
  An example (actually the smallest for which our generalization is nontrivial) to bear in mind is the real reflection group~$W$ of type~$H_2 = I_2(5)$ shown in Figure~\ref{fig:H2}.
  It is the dihedral group of order~$10$ consisting of all symmetries of a regular pentagon.
  The degrees of~$W$ are~$2$ and~$5$, and the Coxeter number is thus $h=5$.
  Set~$s$ and~$t$ to be the reflections through the boundary hyperplanes of a given chamber in the reflection arrangement, and let $c = st$ be a corresponding Coxeter element.
  It is given by a rotation of angle $2\pi/5$.
  Its eigenvalues are~$\zeta$ and $\zeta^4=\zeta^{-1}$, where~${\zeta:=e^{\frac{2i\pi}{5}}}$.
  Observe that $c^2$ is again regular of order~$h$. It is a rotation of angle $4\pi/5$ and has eigenvalues~$\zeta^2$ and~${\zeta^3=\zeta^{-2}}$.
  In particular,~$c$ and~$c^2$ are not conjugate in~$W$.
  In the classical notion of Coxeter elements,~$c$ and~$c^4$ are Coxeter elements (both can be written as products of reflections through the walls of a chamber of the reflection arrangement), while~$c^2$ and~$c^3$ are not.
  Observe though that $c^2 = sts \cdot t$, that the reflections $sts$ and $t$ also generate~$W$, and that $(W,\{sts,t\})$ is again a Coxeter system, isomorphic --- but not conjugate --- to the Coxeter system $(W,\{s,t\})$.
  So $c^2$ and $c^3$ are still products of the generators of a Coxeter system, but this Coxeter system $(W,\{sts,t\})$ does not come from the walls of a chamber, as in the usual Coxeter generators of a reflection group. 
  
    \begin{figure}[t]
    \begin{center}
      \begin{tikzpicture}[scale=1.2]
        \tikzstyle{hp}=[line width=1pt,opacity=0.3,black]
        \tikzstyle{hp2}=[line width=1.5pt,opacity=0.4,black]
        \tikzstyle{hp3}=[line width=1.5pt, opacity=0.4,black,dashed]
        \coordinate (X) at (0,0);
        \polygon{(X)}{obj}{5}{2.5}{,,,,};
        \fill[fill=black!8!white]
            (0,0) -- ($1.25*(obj1)$) arc (18:54:3.125);
        \draw[red,->]
            ($1.1*(obj1)$) arc (18:90:2.75);
        \node (c) at ($0.6*($(obj5)-(obj3)$)$) {\scriptsize $c$};

        \draw[red,->]
            ($1.2*(obj1)$) arc (18:162:3);
        \node (c2) at ($0.68*($(obj4)-(obj2)$)$) {\scriptsize $c^2$};

        \polygon{(X)}{obj}{5}{2.5}{,,,,};
        \draw[hp2] ($1.3*(obj1)$) to ($-1.3*(obj1)$);
        \draw[hp3] ($1.3*(obj1)$) to ($-1.3*(obj1)$);
        \draw[hp] ($1.3*(obj2)$) to ($-1.3*(obj2)$);
        \draw[hp2] ($1.3*(obj3)$) to ($-1.3*(obj3)$);
        \draw[hp] ($1.3*(obj4)$) to ($-1.3*(obj4)$);
        \draw[hp3] ($1.3*(obj5)$) to ($-1.3*(obj5)$);
        \draw[line width=1pt,black] (obj1) to (obj2) to (obj3) to (obj4) to (obj5) to (obj1);

        \filldraw[black] (0,0) circle(2pt);

        \node (e)  at ($0.35*($(obj1)-(obj3)$)$) {\scriptsize $\one$};

        \node (s) at ($0.35*($(obj5)-(obj3)$)$) {\scriptsize $s$};
        \node (t) at ($0.35*($(obj1)-(obj4)$)$) {\scriptsize $t$};

        \node (st) at ($0.35*($(obj5)-(obj2)$)$) {\scriptsize $st$};
        \node (ts) at ($0.35*($(obj2)-(obj4)$)$) {\scriptsize $ts$};

        \node (sts) at ($0.35*($(obj4)-(obj2)$)$) {\scriptsize $sts$};
        \node (tst) at ($0.35*($(obj2)-(obj5)$)$) {\scriptsize $tst$};

        \node (stst) at ($0.35*($(obj4)-(obj1)$)$) {\scriptsize $stst$};
        \node (tsts) at ($0.35*($(obj3)-(obj5)$)$) {\scriptsize $tsts$};

        \node (ststs) at ($0.35*($(obj3)-(obj1)$)$) {\scriptsize $ststs$};

        \node (Ht) at ($1.4*(obj1)$) {\scriptsize $H_t$};
        \node (Hs) at ($-1.4*(obj3)$) {\scriptsize $H_s$};
        \node (Hsts) at ($1.4*(obj5)$) {\scriptsize $H_{sts}$};
        \node (Htst) at ($-1.42*(obj4)$) {\scriptsize $H_{tst}$};
        \node (Hststs) at ($-1.4*(obj2)$) {\scriptsize $H_{ststs}$};
      \end{tikzpicture}
    \end{center}
    \caption{The symmetry group of the regular pentagon.}
    \label{fig:H2}
  \end{figure}
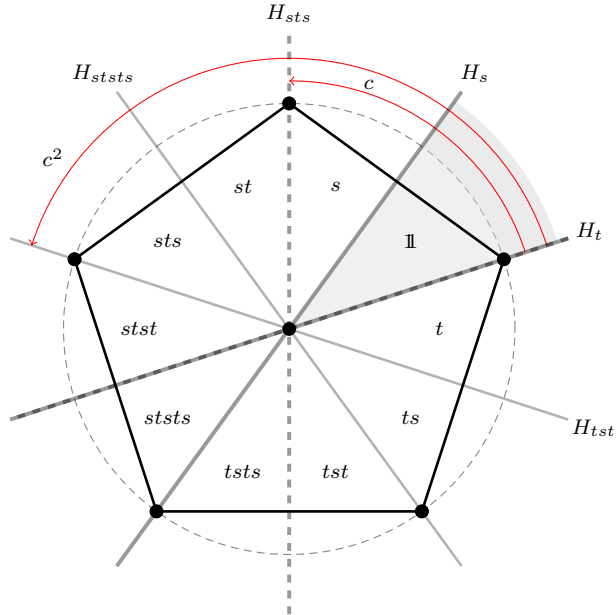
\end{example}

For an irreducible well-generated group~$W$ and a Coxeter element~$c$ in~$W$, we define the \defn{poset of~$W$-noncrossing partitions} $NC(W,c)$ as in~\eqref{eq:ncp}.
Since conjugation by elements of~$W$ preserves the set of reflections~$R$, it also preserves the absolute length~$\ell_R$ and respects the absolute order $\leq_R$.
Hence, whenever two Coxeter elements~$c$ and~$c'$ are~$W$-conjugate, the posets $\NC{c}$ and $\NC{c'}$ are isomorphic.
T.~A.~Springer showed in~\cite[Theorem 4.2]{Spr1974} that, for a fixed~$\zeta$ in $\CC$, the~$\zeta$-regular elements in~$W$ form a single~$W$-conjugacy class; this is recalled in Theorem~\ref{Springer-theory-facts}\eqref{eq:springer1} below.
It is known to experts that the poset structure on $\NC{c}$ does {\bf not} depend on the choice of the Coxeter element~$c$, even in the more general notion of Coxeter elements in which there is not necessarily a single conjugacy class.
But this seems to have only been mentioned so far in~\cite{CS2012}, where the poset isomorphisms were checked explicitly using a computer.
As a result of the considerations in this paper, we obtain a conceptual reason for the existence of a poset isomorphism between~$\NC{c}$ and~$\NC{c'}$ for two regular elements of order~$h$ that are {\bf not}~$W$-conjugate, see Corollary~\ref{cor:ncp} below.

\medskip

\subsection{Characterizations of Coxeter elements}
\label{ssec:charac-coxeter}
The first main result of this paper, Theorem~\ref{thm:main} below, provides several alternative characterizations of Coxeter elements.
We show that the set of Coxeter elements forms a single orbit in~$W$ for the action of \defn{reflection automorphisms}, that is, those automorphisms of~$W$ that preserve the set~$R$ of reflections. 
This allows one to characterize Coxeter elements using only $e^{2i\pi/h}$-regular elements and reflection automorphisms as described in characterization~\eqref{def4} below.
It also implies that Coxeter elements are exactly the products of generators of~$W$ for some well-behaved generating subsets of reflections; see characterization~\eqref{def6}, for which we need to explain beforehand the terminology.
We call such a well-behaved generating set of $W$ a \emph{regular generating set}: it is a minimal generating set of reflections having some additional properties as will be made precise in Section~\ref{ssec:pres}.
We will see in Proposition~\ref{prop:standardpresreg} that any irreducible well-generated group admits a regular generating set.
Note that most explicit presentations for~$W$ using diagrams ``\`a la Coxeter'' feature a regular generating set.
We say that two generating sets for~$W$ are \emph{isomorphic} if there exists a bijection between them which extends to an automorphism of~$W$.

Moreover we obtain the purely combinatorial characterization~\eqref{def5} for Coxeter elements in finite groups that admit a \emph{generalized Coxeter system} as defined in Section~\ref{ssec:shephard}.
We emphasize that the irreducible complex reflection groups admitting a generalized Coxeter system are exactly the real reflection groups and the Shephard groups, and that for a real reflection group, a generalized Coxeter system is a Coxeter system in the usual sense.

\begin{theorem}
  \label{thm:main}
  Let~$W$ be an irreducible well-generated complex reflection group with Coxeter number~$h$, and let~$c \in W$.
  The following statements are equivalent.
  \begin{enumerate}[(i)]
    \item~$c$ is a Coxeter element (i.e.,~$c$ is regular of order~$h$); \label{def1}
    \item $c=w^p$ for an $e^{2i\pi/h}$-regular element~$w$ and an integer~$p$ coprime to~$h$; \label{def3}
    \item~$c$ has an eigenvalue of order~$h$; \label{def2}
    \item $c=\psi(w)$ for an $e^{2i\pi/h}$-regular element~$w$ and a reflection automorphism~$\psi$. \label{def4}
  \end{enumerate}
  Fix a regular generating set~$S_0$ for~$W$.
  Then the following statement is as well equivalent.
  \begin{enumerate}[(i)]
    \setcounter{enumi}{4}
    \item There exists a generating set $S$ contained in $R$ and isomorphic to $S_0$, such that~$c$ is the product (in some order) of the elements in~$S$. \label{def6}
  \end{enumerate}
  If~$W$ admits a generalized Coxeter system, the following statement is as well equivalent.
  \begin{enumerate}[(i)]
    \setcounter{enumi}{5}
  \item There exists a subset~$S$ of the set~$R$ of reflections in~$W$ such that~$(W,S)$ is a generalized Coxeter system and~$c$ is the product (in some order) of the elements in~$S$. \label{def5}
  \end{enumerate}
\end{theorem}

The main consequence of Theorem~\ref{thm:main} is that characterization~\eqref{def4} allows one to transfer properties that are known for ``classical'' Coxeter elements to the more general notion of Coxeter elements.
It shows moreover that the two natural generalizations of the usual definition of Coxeter elements (one in the real setting, the other for the complex setting) coincide in the real case.
Property~\eqref{def5} can indeed be taken as a natural generalized definition of Coxeter elements in finite real reflection groups. 
We will see that this generalization brings new elements whenever the Coxeter group is noncrystallographic (see Remark~\ref{rk:weyl}).
Note that a slightly more restrictive version of this definition already appears in D.~Bessis' work on the dual braid monoid of Coxeter groups, see~\cite[Definition~1.3.2]{Bessis1}.
This generalized definition for Coxeter elements was also used for general finite or infinite Coxeter groups in~\cite{BDSW2014}. 
Table~\ref{table:cox} records the ``classical'' and the more general definition of Coxeter elements for an irreducible well-generated group, either real or complex.
\begin{table}[ht]
  \begin{tabular}{C{2.5cm}|C{5.5cm}|C{5.5cm}}
    & ``classical'' definition & general definition \\[5pt]
    \hline
    &&\\
    \vspace{2pt}$W$ real &
      product of the reflections through the walls of some chamber &
      product of elements of~$S$ for some Coxeter system $(W,S)$ with $S\subseteq R$ \\[5pt]
    \hline
    &&\\
   ~$W$ complex &
      $\displaystyle{e^{\frac{2i\pi}{h}}}$-regular &
      regular of order~$h$\\
    &&
  \end{tabular}
  \caption{Different notions of Coxeter elements in a reflection group~$W$}
  \label{table:cox}
\end{table}

The proof of Theorem~\ref{thm:main} will be derived as follows.
The assertions~\eqref{def1}$\Leftrightarrow$\eqref{def3} and~\eqref{def1}$\Rightarrow$\eqref{def2} are direct consequences of the definition of Springer's regularity.
The implication~\eqref{def2}$\Rightarrow$\eqref{def1} comes from a counting argument using Pianzola-Weiss' formula, and can be found in~\cite[Theorem~32-2C]{Kan2001} for the real case and in~\cite[Proposition~2.1]{CS2012} for the general case.
In this paper, we prove the three remaining characterizations~\eqref{def4},~\eqref{def6} and~\eqref{def5}.
The characterization~\eqref{def4} is a consequence of Proposition~\ref{prop:refl-autom} below, which we prove in Section~\ref{sec:refl-autom}.
This characterization is the heart of the present paper.
After showing that any irreducible well-generated group admits a regular generating set in Section~\ref{ssec:pres}, characterization~\eqref{def6} will follow easily from~\eqref{def4}. 
Characterization~\eqref{def5} for real reflection groups and Shephard groups is then proven in Section~\ref{ssec:charac-cox}.
The proof relies on characterization~\eqref{def6} and on a rigidity property of the generalized Coxeter presentations of these groups.

\medskip

The statement below is a reformulation of characterization~\eqref{def4} in Theorem~\ref{thm:main}. 
We denote by $\Aut_R(W)$ the group of reflection automorphisms of~$W$.

\begin{proposition}
  \label{prop:refl-autom}
  The action of $\Aut_R(W)$ on~$W$ preserves the set of Coxeter elements, and is transitive on it.
\end{proposition}

In particular, for any two Coxeter elements~$c$ and~$c'$, there is a reflection automorphism~$\psi$ mapping~$c$ to $c'$. 
Since~$\psi$ sends reflections to reflections, the absolute length~$\ell_R$ and the absolute order~$\leq_R$ are respected by~$\psi$. 
Hence,~$\psi$ restricts to a poset isomorphism from the interval $[\one,c]_W$ to the interval $[\one,c']_W$ in the absolute order, and we obtain the following corollary.

\begin{corollary}
  \label{cor:ncp}
  Let~$W$ be an irreducible well-generated complex reflection group, and let~$c$ and~$c'$   be two Coxeter elements.
  Then the two posets $\NC{c}$ and $\NC{c'}$   are isomorphic.
\end{corollary}

This corollary implies that all properties of the poset of noncrossing partitions that have only been proven for the classical definition of Coxeter elements also hold for the more general definition.
In particular, for a Coxeter element~$c \in W$, the poset $\NC{c}$ is an EL-shellable, self-dual lattice whose elements are counted by the \defn{$W$-Catalan number} $\prod_{i=1}^{n}\frac{d_i+h}{d_i}$.
We refer to~\cite{Arm2006,BessisR,KM2008,Muh2011} for these and several other properties of the noncrossing partitions that were so far only proven for the restricted definition of Coxeter elements.

More generally, Proposition~\ref{prop:refl-autom} implies that all the properties of Coxeter elements relative only to the combinatorics of the group~$W$ equipped with its set of generators $R$, do not depend on the choice of a Coxeter element.
For example, the transitivity of the Hurwitz action of the $n$-strands braid group~$B_n$ on the reduced $R$-decompositions of $e^{2i\pi/h}$-regular elements (see~\cite[Definition~6.19, Proposition~7.6]{Bessis2}) implies the same property for all regular elements of order~$h$.
This property was proven for Coxeter elements of infinite Coxeter groups as well in~\cite[Theorem~1.3]{BDSW2014}.

\begin{remark}
  Theorem~\ref{thm:main}\eqref{def5} provides a purely combinatorial description of Coxeter elements in reflection groups admitting generalized Coxeter systems.
  Nevertheless, this property is rather difficult to check for a given element.
  We do not know of any better purely combinatorial description, nor do we know of any combinatorial description for Coxeter elements in arbitrary well-generated groups.
  Coxeter elements have a number of simple properties: they have absolute length~$n$ (the rank of the group), order~$h$, and the Hurwitz action is transitive on the set of reduced~$R$-decompositions of a Coxeter element.
  However, these properties are not sufficient to characterize Coxeter elements.
  For example, there exist elements in the group of type~$D_4$ that are of absolute length~$4$ and for which the Hurwitz action is transitive on reduced $R$-decompositions, but which do not have a primitive~$6$-th root of unity as an eigenvalue\footnote{We thank Patrick Wegener for pointing out the existence of such elements.}.
  Similarly, there exist elements in type~$B_6$ that are of absolute length~$6$ and order~$12$, but which also do not have a primitive~$6$-th root of unity as an eigenvalue.
\end{remark}

\subsection{Conjugacy classes of Coxeter elements}
\label{ssec:galois-conj}

Our second main theorem concerns the connection between Coxeter elements and the field of definition of~$W$.
Recall that the \defn{field of definition}~$K_W$ is the subfield of~$\CC$ generated by the traces of the elements in $W\subseteq \GL(V)$.
It is known, see e.g.~\cite[Proposition~7.1.1]{benson}, that the representation~$V$ of~$W$ can be realized over~$K_W$.
We thus assume from now on that $W\subseteq\GL_n(K_W)$.
We also denote by~$\Gamma_W:=\Gal(K_W/\QQ)$ the Galois group of the field extension~$K_W$ over~$\QQ$.
We next describe an action of~$\Gamma_W$ on the conjugacy classes of Coxeter elements, and show that this action is simply transitive, see Theorem~\ref{thm:galois}.

\medskip

The group~$\Gamma_W$ naturally acts on $\GL_n(K_W)$ by Galois conjugation of the matrix entries. For~$\gamma\in \Gamma_W$, we denote by $\bar{\gamma}$ the associated automorphism of $\GL_n(K_W)$.
However, this action does not necessarily preserve~$W$, so we cannot always associate an automorphism of~$W$ to~$\gamma$.
For  example, if~$W$ is the dihedral group of type~$H_2=I_2(5)$, then~$K_W=\QQ(\sqrt{5})$ and the Galois group~$\Gamma_W$ has order~$2$.
But one can check that the only involutive automorphisms of~$I_2(5)$ are inner automorphisms, i.e., are given by conjugation by an element of~$W$.

Nevertheless, $\bar{\gamma}(W)$ is obviously again a complex reflection group, and it turns out that it is always conjugate to~$W$ in $\GL_n(\CC)$ (cf. \cite[Theorem~8.32]{LehrerTaylor}; see also Corollary~\ref{cor:galoisconj} below).
Let $g\in\GL_n(\CC)$ be such that $\bar{\gamma}(W)=gWg^{-1}$.
We then obtain from $\gamma$ an automorphism of~$W$ given by $w \mapsto g^{-1}\bar{\gamma}(w)g$. 
We call an automorphism obtained this way a \defn{Galois automorphism} of~$W$ attached to $\gamma$. 
These automorphisms have been studied in detail by I.~Marin and J.~Michel in \cite{MM2010}, and we will rely on one of their results, see Proposition~\ref{prop:marinmichel}.

\medskip

We first record the following straightforward properties.
\begin{enumerate}[(i)]
  \item The character of a Galois automorphism attached to $\gamma$ (seen as a representation of~$W$) is given by $w\mapsto \gamma(\tr_V(w))$.

  \item There are several choices for such an automorphism but one can pass from one to another via conjugation by some element in the normalizer $N_W = N_{\GL_n(K_W)}(W)$.\label{it:normalizer}

  \item Any Galois automorphism of~$W$ is a reflection automorphism.

  \item Let~$\psi$ be a reflection automorphism of~$W$. 
  Then~$\psi$ is a Galois automorphism attached to $\gamma$ if and only if the character of~$\psi$ (seen as a representation of~$W$) is the image by~$\gamma$ of the character~$V$, i.e.,
  \[
    \forall w \in W, \tr_V(\psi(w))=\gamma(\tr_V(w)).
  \]
\end{enumerate}

We do not necessarily have a natural action of~$\Gamma_W$ on~$W$, but using~\eqref{it:normalizer}, we get an action of~$\Gamma_W$ on the set of~$N_W$-conjugacy classes of elements in~$W$.
Note that two regular elements which are~$N_W$-conjugate are also~$W$-conjugate. 
Indeed, if~$w'=awa^{-1}$, with~$a\in N_W$, and if~$v$ is a~$\zeta$-regular eigenvector, then~$av$ is a~$\zeta$-eigenvector for~$w'$. 
Moreover it is easy to see that the action of~$N_W$ sends a reflecting hyperplane of~$W$ to another one, so preserves the regular vectors. 
Thus~$av$ is a~$\zeta$-regular eigenvector for~$w'$, so~$w'$ is~$\zeta$-regular, and is~$W$-conjugate to~$w$ by Springer's theorem (see Theorem~\ref{Springer-theory-facts}\eqref{eq:springer1}).

Therefore,~$N_W$-conjugacy classes of regular elements (and in particular of Coxeter elements) are the same as~$W$-conjugacy classes\footnote{In fact,~\cite[Theorem~1.3]{MM2010} actually implies the stronger property that~$\Gamma_W$ always acts on the set of~$W$-conjugacy classes.}.
Denote by~$\Cox(W)$ the set of conjugacy classes of Coxeter elements. By Proposition~\ref{prop:refl-autom}, reflection automorphisms stabilize the set of Coxeter elements, so the action of~$\Gamma_W$ stabilizes~$\Cox(W)$.
Theorem~\ref{thm:galois} below states that the action of~$\Gamma_W$ on~$\Cox(W)$ is simply transitive. 
The transitivity follows from the following result from~\cite{MM2010}.

\begin{proposition}[{cf.~\cite[Theorem 1.2]{MM2010}}]
  \label{prop:marinmichel}
  Any reflection automorphism of an irredu\-cible complex reflection group~$W$ is a Galois automorphism (attached to some~$\gamma\in \Gamma_W$).
\end{proposition}

Combined with Proposition~\ref{prop:refl-autom}, this implies that the action of~$\Gamma_W$ is transitive on~$\Cox(W)$. 
The proof of the \emph{simple} transitivity undertaken in Section~\ref{sec:galois} exhibits several other properties related to the field of definition, that we collect in Theorem~\ref{thm:galois}.
We use the following notation:
\begin{itemize}
  \item $m_1=d_1-1,\dots, m_n=d_n-1$ are the \defn{exponents} of~$W$,
  \item $\varphi(j)$ (for $j\in \NN$) is the number of integers in $\{1,\dots,j\}$ that are coprime to~$j$, and
  \item $\varphi_W(j)$ is the number of integers coprime to~$j$ among the set of exponents $\{m_1,\dots,m_n\}$.
\end{itemize}

For an irreducible, well-generated group~$W$ with Coxeter number~$h$, we also denote by~$G_W$ the setwise stabilizer of $\big\{\zeta^{m_1},\ldots,\zeta^{m_n}\big\}$ in the Galois group $\Gal(\QQ(\zeta)/\QQ)$, where~$\zeta$ is a primitive~$h$-th root of unity.
\begin{theorem}
  \label{thm:galois}
  Let~$W$ be an irreducible well-generated complex reflection group,~$K_W$ be its field of definition, and~$\Cox(W)$ be the set of conjugacy classes of Coxeter elements of~$W$.
  The following properties hold:
\begin{enumerate}[(i)]
  \item the transitive action of~$\Gamma_W = \Gal(K_W/\QQ)$ on~$\Cox(W)$ is free; \label{eq:freeaction0}
  \item $\displaystyle{[K_W : \QQ]= \left | \Cox(W) \right| = \varphi(h) / \varphi_W(h)}$; \label{eq:degreeoffieldextension0}
  \item  $K_W$ is equal to the fixed field $\QQ(\zeta)^{G_W}$; \label{eq:invariantsubfield0}
  \item $K_W$ is generated by the coefficients of the characteristic polynomial of any Coxeter element of~$W$. \label{eq:generatedbycoefficients0}
\end{enumerate}
\end{theorem}

\begin{remark}
  The characterization of~$K_W$ in~\eqref{eq:invariantsubfield0} (or in~\eqref{eq:generatedbycoefficients0}, which is easily seen to be equivalent) has already been obtained by G.~Malle in~\cite[Theorem~7.1]{malle}, his proof using a case-by-case check via the classification. 
  We found it independently by other means but we also need a case-by-case analysis. 
\end{remark}

The equality $\left | \Cox(W) \right| = \varphi(h) / \varphi_W(h)$ is a direct consequence of Springer's theory (see Lemma~\ref{lem:nccox}) and is included in the theorem for the sake of clarity. 
In Section~\ref{sec:galois}, we first prove that the four statements are equivalent; the proof for this is case-free, except for the use of Proposition~\ref{prop:refl-autom}.
The theorem is then derived by checking via the classification that the equality $[K_W : \QQ]= \varphi(h) / \varphi_W(h)$ is satisfied for any irreducible well-generated group.

\begin{remark}
  \label{rk:weyl}
  Recall that a complex reflection group is a finite Weyl group if and only if its field of definition is~$\QQ$.
  Theorem~\ref{thm:galois}\eqref{eq:freeaction0} thus implies that all regular elements of order~$h$ are $e^{2i\pi/h}$-regular if and only if~$W$ is a Weyl group.
  We also recover the well-known fact that for Weyl groups, the equality $\varphi_W(h)=\varphi(h)$ holds, see~\cite[Proposition~3.20]{Humphreys}.
\end{remark}

\begin{remark} 
  \label{rk:bad}
  The intriguing relation in Theorem~\ref{thm:galois}\eqref{eq:degreeoffieldextension0} between the field of definition and the residues of the exponents modulo~$h$ yields the question of what happens for badly-generated groups.
  There are~$8$ of them in the exceptional types.
  For all but~$G_{15}$, the highest invariant degree~$d_n$ is still a regular number, so we could define a Coxeter element as a regular element of order $d_n$ as for well-generated groups. 
  For badly-generated groups in the infinite series,~$d_n$ is regular only for the groups of type $G(2d,2,2)$.
  In these cases, the equality $|\Cox(W) |=\varphi(d_n) / \varphi_W(d_n)$ still holds for the same reasons as for well-generated groups.
  Moreover, we have
  \begin{enumerate}[(i)]
    \item for $G(2d,2,2)$, $G_7$, $G_{11}$, $G_{12}$, $G_{19}$, $G_{22}$ and $G_{31}$, one still has $[K_W:\QQ]=\varphi(d_n) / \varphi_W(d_n)$, whereas \label{eq:bad1}
    \item for $G_{13}$ (and also for~$G_{15}$),  one has $[K_W:\QQ]=2 \varphi(d_n) / \varphi_W(d_n)$.
  \end{enumerate}
  We do not know of an explanation for these observations. 
  This actually implies that for all the groups listed in~\eqref{eq:bad1}, Theorem~\ref{thm:galois} still holds.
  This will be explained in Remark~\ref{rk:bad2}.
\end{remark}

\subsection{Regular elements and reflection automorphisms}
\label{ssec:regular}
In the final Section~\ref{sec:reg} we extend some of our results to regular elements of arbitrary order. 
In particular, statement~\eqref{it:trans-reg} below is a generalized version of Proposition~\ref{prop:refl-autom}.

\begin{theorem}
  \label{thm:refl-autom-reg}
  Let~$W$ be an irreducible complex reflection group, $K_W$ be its field of definition, and~$d$ be a regular number for~$W$. 
  Denote by~$\Cox_d(W)$ the set of conjugacy classes of regular elements of order~$d$. 
  Then
  \begin{enumerate}[(i)]
  \item the action of reflection automorphisms on~$W$ preserves the set of regular elements of order~$d$, and is transitive on it;\label{it:trans-reg}
  \item the natural action of $\Gamma_W = \Gal(K_W/\QQ)$ on $\Cox_d(W)$ is transitive;\label{it:galois-trans}
  \item the cardinality of $\Cox_d(W)$ is $\varphi(d)/\varphi_W(d)$;\label{it:card-conj}
  \item the integer $\varphi(d)/\varphi_W(d)$ divides $[K_W:\QQ]$.\label{it:divides}
  \end{enumerate}
\end{theorem}

Unlike for Coxeter elements, in the general case the action in statement~\eqref{it:galois-trans} may be not free.
In that case, $\varphi(d)/\varphi_W(d)$ is a proper divisor of $[K_W:\QQ]$.

\section{Coxeter elements and reflection automorphisms}
\label{sec:refl-autom}

In this section we prove Proposition~\ref{prop:refl-autom}, thus obtaining the characterization~\eqref{def4} in Theorem~\ref{thm:main}.
We first recall in Section~\ref{ssec:background} some more background on complex reflection groups. 
We also record some facts on extensions of Galois automorphisms in Section~\ref{ssec:extension}.
We then prove in Section~\ref{ssec:image} that the action in Proposition~\ref{prop:refl-autom} is well defined, thus establishing the implication $\eqref{def4}\Rightarrow\eqref{def1}$ in Theorem~\ref{thm:main}.
We finally prove in Section~\ref{ssec:trans} the transitivity part of Proposition~\ref{prop:refl-autom} corresponding to the implication $\eqref{def1}\Rightarrow\eqref{def4}$.

\subsection{Background on complex reflection groups}
\label{ssec:background}

We first need to recall some facts from the classification of irreducible complex reflection groups.
Shephard-Todd's classification~\cite{ShephardTodd} provides that the Shephard-Todd types $G(de,e,n)$ and $G_4,G_5,\ldots,G_{37}$, except for obvious coincidences, contain \defn{exactly one} representative of each $\GL_n(\CC)$-conjugacy class of irreducible reflection groups.
This classification proof was reworked more recently by G.~I.~Lehrer and D.~E.~Taylor in~\cite[Theorem~8.29]{LehrerTaylor}.
Moreover, part of this classification asserts that the groups in the infinite family $G(de,e,n)$ for $n \geq 2$ are the only irreducible reflection groups that act \emph{imprimitively}\footnote{An \defn{imprimitive} action is one where there is a nontrivial direct sum decomposition $V=\oplus_{i=1}^t V_i$ respected by~$W$, in the sense that for each~$w \in W$, there is a permutation $\sigma$ of   $\{1,2,\ldots,t\}$ for which $w(V_i) = V_{\sigma(i)}$.} on $V=\CC^n$, while the exceptional irreducible groups $G_4, G_5, \ldots,G_{37}$ all act \emph{primitively}.

\medskip

We record the following observation, that was communicated to us by D.~Bessis\footnote{Private communication with C.~Stump at the conference \emph{Hyperplane Arrangements: combinatorial and geometric aspects} of the DFG Priority Programms \lq\lq Representation Theory\rq\rq\ and \lq\lq Algorithmic and Experimental Methods in Algebra, Geometry and Number Theory\rq\rq, February 2013 in Bochum, Germany.}.
\begin{proposition}
  \label{prop:observation}
  Let $W,W'\subseteq \GL_n(\CC)$ be two irreducible complex reflection groups.
  Suppose that
  \begin{itemize}
  \item~$W$ and $W'$ are either both primitive, or both imprimitive;
  \item~$W$ and $W'$ have the same multiset of degrees, and the same multiset of codegrees.
  \end{itemize}
  Then~$W$ and $W'$ have the same Shephard-Todd type. In particular, they are conjugate in~$\GL_n(\CC)$.
\end{proposition}

As far as we know, this observation has not been given a conceptual explanation. 
The proof of Proposition~\ref{prop:observation} simply goes by checking that, in the list of irreducible groups, there are no coincidences both of the multisets of degrees and of the multiset of codegrees:
\begin{itemize}
  \item within the imprimitive family $G(de,e,n)$ (for $de\neq 1$), nor

  \item within the primitive  groups (type A and exceptional groups, see for example the tables in \cite{BMR98}). 
\end{itemize}

\begin{remark}
  If one assumes in Proposition~\ref{prop:observation} that~$W$ and~$W'$ are well-generated, then the multisets of codegrees do not need to be checked (since in this case the degrees determine the codegrees). 
  In the general case, checking the codegrees is necessary because of the coincidences of degrees between $G_{10}$ and $G_{15}$, and between $G_8$ and $G_{13}$.
\end{remark}

\begin{corollary}
  \label{cor:galoisconj}
  Let $W \subseteq \GL_n(\CC)$ be an irreducible complex reflection group, and~$K$ be an algebraic extension of $\QQ$ such that $W\subseteq\GL_n(K)$. 
  For $\gamma\in \Gal(K/\QQ)$, denote by $\bar{\gamma}$ the associated automorphism of $\GL_n(K)$.
  Then for any $\gamma\in \Gal(K/\QQ)$, the reflection group $\bar{\gamma}(W)$ is conjugate to~$W$ in $\GL_n(\CC)$.
\end{corollary}

\begin{proof}
  The reflection group $\bar{\gamma}(W)$ has the same degrees as~$W$ (since they are the degrees of the fundamental invariant polynomials), and the same codegrees as well (use for example the last formula in~\cite[Appendix~D.2]{LehrerTaylor}). 
  Moreover, $\bar{\gamma}(W)$ is primitive if and only if~$W$ is primitive. 
  So the proof follows from Proposition~\ref{prop:observation}.
  Note that this fact can also be deduced from a precise treatment of the classification, as mentioned in~\cite[Theorem~8.32]{LehrerTaylor}.
\end{proof}

Next, we record two facts from Springer's theory of regular elements 
for later convenience.

\begin{theorem}[{cf.~\cite[Theorem 4.2 (iv,v)]{Spr1974}}]
\label{Springer-theory-facts}
  Let~$\zeta$ be a root of unity and let~$w$ and~$w'$ be~$\zeta$-regular elements in a complex reflection group~$W$.
  Then
\begin{enumerate}[(i)]
\item $w$ and~$w'$ are~$W$-conjugate; \label{eq:springer1}
\item the eigenvalues of~$w$ (and of~$w'$) are given by $\zeta^{-m_1},\ldots,\zeta^{-m_n}$, where~$m_1,\dots, m_n$ are the exponents of~$W$.\label{eq:springer2}
\end{enumerate}
\end{theorem}

\subsection{Galois theory}
\label{ssec:extension}

We recall the following property.

\begin{proposition}
  \label{Galois-theory-prop}
  Given a field~$K$ which is a normal extension of its prime field $k_0$, and any subfield $k \subseteq K$, every field automorphism of~$k$ extends to a field automorphism of~$K$.
\end{proposition}

\begin{proof}
  Any automorphism $\gamma: k \to k$ fixes the prime field $k_0$ pointwise.
  Therefore one can use the usual isomorphism extension theorem to extend $\gamma$ to an element $\delta\in \Gal(\bar{K} / k_0 )$, where $\bar{K}$ denotes the algebraic closure of~$K$, which is also the algebraic closure of~$k$ (since~$K$ is normal and hence algebraic over $k_0$, one has that~$K$ is algebraic over~$k$).
  But then $\delta$ sends any element $z\in K$ to another root $\delta(z)$ of its minimal polynomial over~$k_0$, which still lies in~$K$ because $K / k_0$ is normal.
  Hence $\delta$ restricts to an automorphism $K \to K$, extending~$\gamma$.
\end{proof}

This implies two facts on the field of definition~$K_W$ to be used later.

\begin{corollary}
\label{Galois-corollary}
  Let~$W$ be a complex reflection group and~$\zeta$ be a root of unity.
  \begin{enumerate}[(i)]

    \item Every element of $\Gal(K_W/\QQ)$ extends to an element of $\Gal(K_W(\zeta)/\QQ)$. \label{it:extension1}

    \item Every element of $\Gal(\QQ(\zeta)/\QQ)$ extends to an element of $\Gal(K_W(\zeta)/\QQ)$.
      \label{it:extension2}

  \end{enumerate}
\end{corollary}

\begin{proof}
  Since~$W$ is a finite group, the field of definition $K_W$, and thus $K_W(\zeta)$ as well, are generated over $\QQ$ by a finite number of roots of unity. 
  So $K_W(\zeta)$ is an abelian, hence Galois, extension of $\QQ$. 
  Thus we can apply Proposition~\ref{Galois-theory-prop} to $K_W(\zeta)$, and assertions~\eqref{it:extension1} and~\eqref{it:extension2} follow.
\end{proof}

\subsection{Images of Coxeter elements by reflection automorphisms}
\label{ssec:image}

To prove that the action in Proposition~\ref{prop:refl-autom} is well-defined, we have to show that for any Coxeter element~$c \in W$ and any reflection automorphism~$\psi \in \Aut_R(W)$, the element~$\psi(c) \in W$ is again a Coxeter element.
By the equivalence \eqref{def1}$\Leftrightarrow$\eqref{def2} of Theorem~\ref{thm:main}, it is enough to check that~$\psi(c)$ has an eigenvalue of order~$h$.
From Proposition~\ref{prop:marinmichel}, we know that~$\psi$ is a Galois automorphism attached to some $\gamma\in \Gamma_W$.
Moreover, from the construction of Galois automorphisms in Section~\ref{sec:intro},~$\psi(c)$ and~$\bar{\gamma}(c)$ are $\GL_n(\CC)$-conjugate, so we just need to check that~$\bar{\gamma}(c)$ has an eigenvalue of order~$h$.

Let~$\zeta$ be an eigenvalue of order~$h$ for~$c$; then all eigenvalues of~$c$ live in $\QQ(\zeta)$, by Theorem~\ref{Springer-theory-facts}\eqref{eq:springer2}.
It now follows from Corollary~\ref{Galois-corollary}\eqref{it:extension1} that we can extend the field automorphism~$\gamma \in \Gal(K_W / \QQ)$ to an automorphism~$\delta \in \Gal(K_W(\zeta) /\QQ)$.
Note that $\delta(\zeta)$ must be another root of the cyclotomic polynomial $\Phi_h$, that is,
$\delta(\zeta)=\zeta^p$ for some~$p$ coprime to~$h$.
But the eigenvalues of~$\bar{\gamma}(c)$ are simply the images by~$\delta$ of the ones of~$c$. 
So $\delta(\zeta)=\zeta^p$ is an eigenvalue of $\bar{\gamma}(c)$, which concludes the proof.

\subsection{Transitivity of reflection automorphisms on Coxeter elements}
\label{ssec:trans}

To complete the proof of Proposition~\ref{prop:refl-autom}, it remains to prove that for any two Coxeter elements~$c$ and~$c'$, there is a reflection automorphism~$\psi$ such that $\psi(c) = c'$.
Assume that the Coxeter elements~$c$ and~$c'$ have regular eigenvalues~$\zeta$ and $\zeta'= \zeta^p$, respectively, for some~$p$ coprime to~$h$.
Denote by~$K$ the field $K_W(\zeta)$.
Corollary~\ref{Galois-corollary}\eqref{it:extension2} implies that there is an automorphism~$\delta \in \Gal(K/\QQ)$  extending\footnote{One can observe on the classification (see e.g.~\cite[Theorem~7.1]{malle}) that for any well-generated group, the field~$\QQ(\zeta)$ already contains~$K_W$; but we do not need this here.} the automorphism of $\Gal(\QQ(\zeta)/\QQ)$ that sends~$\zeta$ to $\zeta'= \zeta^p$.

Let us denote by $\bar{\delta}$ the automorphism of $\GL_n(K)$ associated to~$\delta$ (defined by Galois conjugation of the matrix entries).
Since $W\subseteq \GL_n(K)$, we obtain a (potentially) new reflection group $\bar{\delta}(W)$. But by Corollary~\ref{cor:galoisconj}, $\bar{\delta}(W)$ is conjugate to~$W$ in $\GL_n(\CC)$.
Let $g \in \GL_n(\CC)$ be such that $\bar{\delta}(W)=gWg^{-1}$.
As $c'$ and $g^{-1}\bar{\delta}(c)g$ are both $\zeta'$-regular elements of~$W$, Springer's Theorem~\ref{Springer-theory-facts}\eqref{eq:springer1} says that they are conjugate in~$W$, say $c' = w^{-1}g^{-1}\bar{\delta}(c)gw$ for some~$w$ in~$W$.
Then the map
\[
  \begin{array}{lrcl}
    \psi: & \GL_n(\CC) & \to & \hspace*{15pt} \GL_n(\CC) \\[7pt]
          & u\hspace*{15pt} & \mapsto & w^{-1}g^{-1}\bar{\delta}(u)gw
  \end{array}
\]
preserves~$W$, and therefore yields an automorphism of~$W$.
Moreover, we have $\psi(c)=c'$ and~$\psi$ is a reflection automorphism since the maps
$\bar{\delta}, g^{-1}(-)g, w^{-1}(-)w$ all send reflections to reflections.

\section{Coxeter elements and Galois automorphisms}
\label{sec:galois}

In this section, we prove Theorem~\ref{thm:galois} concerning the Galois action on conjugacy classes. 
It follows from Proposition~\ref{prop:marinmichel} that one has an action of $\Gamma_W=\Gal(K_W/\QQ)$ on the set of conjugacy classes of regular elements in~$W$.
By Proposition~\ref{prop:refl-autom}, this action preserves the set~$\Cox(W)$ of conjugacy classes of Coxeter elements.
Moreover, the action by reflection automorphisms is transitive on Coxeter elements.
Since any reflection automorphism is a Galois automorphism (Proposition~\ref{prop:marinmichel}), this also implies that the action of~$\Gamma_W$ on~$\Cox(W)$ is transitive. 
We state in the following proposition several properties that are equivalent to the freeness of this action. 
Recall that~$G_W$ denotes the setwise stabilizer of $\big\{\zeta^{m_1},\ldots,\zeta^{m_n}\big\}$ in the Galois group $\Gal(\QQ(\zeta)/\QQ)$, where~$\zeta$ is a primitive~$h$-th root of unity and $m_1,\dots, m_n$ are the exponents of~$W$.

\begin{proposition}
\label{prop:equiv}
  Let~$W$ be an irreducible well-generated group with Coxeter number~$h$, and let~$\zeta$ be a primitive~$h$-th root of unity.
  The following statements are equivalent.
  \begin{enumerate}[(i)]
  \item the action of\/~$\Gamma_W$ on\/ $\Cox(W)$ is free (thus simply transitive); \label{eq:freeaction}
  \item $[K_W : \QQ]= \varphi(h) / \varphi_W(h)$; \label{eq:degreeoffieldextension}
  \item $K_W=\QQ(\zeta)^{G_W}$; \label{eq:invariantsubfield}
  \item $K_W$ is generated by the coefficients of the characteristic polynomial of any Coxeter element of\/~$W$. \label{eq:generatedbycoefficients}
  \end{enumerate}
\end{proposition}

The proof of Proposition~\ref{prop:equiv} will follow from the three lemmas below, which are easy consequences of Springer's Theorem. 

\begin{lemma}
  \label{lem:nccox}
  With the same notation as above, the conjugacy classes of Coxeter elements in~$W$ are counted by
\[\left | \Cox(W) \right| = \frac{\varphi(h)}{\varphi_W(h)}.\]
\end{lemma}

\begin{proof}
  Denote by~$\mu_h'$ the set of primitive~$h$-th roots of unity.
  It follows from Springer's Theorem~\ref{Springer-theory-facts}\eqref{eq:springer1} that any element of~$\mu_h'$ appears as an eigenvalue of exactly one conjugacy class of Coxeter elements.
  So one has
  \begin{align}
    \mu_h'=\bigsqcup_{C\in \Cox(W)} \Sp(C)\cap\mu_h' \label{eq:mu}
  \end{align}
  where~$\Sp(C)$ denotes the set of eigenvalues of any element in the conjugacy class~$C$.
  By Springer's Theorem~\ref{Springer-theory-facts}\eqref{eq:springer2}, the eigenvalues of a~$\zeta$-regular element are $\zeta^{-m_1},\dots, \zeta^{-m_n}$ where $m_1,\dots, m_n$ are the exponents of~$W$.
  Thus for any $C\in \Cox(W)$,
  \[
    \big|\Sp(C)\cap\mu_h'\big|=\Big|\big\{m_i \,|\, i\in\{1,\dots,n\}\text{ and } \gcd(m_i,h) =1 \big\}\Big| =\varphi_W(h).
  \]
  Using Equation~\eqref{eq:mu}, we get the equality $\big|\Cox(W) \big|=\varphi(h)/\varphi_W(h)$.
\end{proof}

\begin{lemma}
  \label{lem:zeta}
  Let~$c$ be a Coxeter element of~$W$ and let~$\zeta$ be a primitive~$h$-th root of unity.
  Then~$\QQ(\zeta)^{G_W}$ is generated by the coefficients of the characteristic polynomial of~$c$.
  In particular,~$\QQ(\zeta)^{G_W}$ is contained in the field of definition~$K_W$.
\end{lemma}

\begin{proof}
  The group~$G_W$ and the field~$\QQ(\zeta)$ do not depend on the choice of the primitive~$h$-th root of unity~$\zeta$, so we can assume that~$c$ is a~$\zeta$-regular element. 
  Then, from Springer's Theorem~\ref{Springer-theory-facts}\eqref{eq:springer2}, the eigenvalues of~$c$ are $\zeta^{-m_1}, \dots, \zeta^{-m_n}$, so the characteristic polynomial of~$c$ verifies
  \[
    \charpol_c(X)= (X-\zeta^{-m_1})(X-\zeta^{-m_2}) \dots (X-\zeta^{-m_n}) = X^n - e_1 X^{n-1} + \dots + (-1)^ne_n,
  \]
  where $e_1,\dots, e_n$ are the elementary symmetric functions on the~$\zeta^{-m_i}$'s.
  Thus an element $\gamma\in \Gal(\QQ(\zeta)/\QQ)$ stabilizes the set $\big\{\zeta^{m_1},\ldots,\zeta^{m_n}\big\}$ if and only if it fixes $e_1,\dots, e_n$. 
  So the group~$G_W$ is $\Gal(\QQ(\zeta)/\QQ(e_1,\dots, e_n))$ and thus $\QQ(\zeta)^{G_W}=\QQ(e_1,\dots, e_n)$.
\end{proof}

\begin{lemma}
  \label{lem:exponents}
  The group~$G_W$ consists of all the automorphisms defined by $\zeta\mapsto \zeta^{-m_i}$, where~$m_i$ is an exponent of~$W$ which is coprime to~$h$.
  In particular, $|G_W|=\varphi_W(h)$.
\end{lemma}

\begin{proof}
  Let $\gamma$ be an element of $\Gal(\QQ(\zeta)/\QQ)$ given by $\zeta\mapsto \zeta^k$ for some $k\in\{0,\dots,h-1\}$ coprime to~$h$. 
  Let~$c$ be a~$\zeta$-regular element. 
  Using Springer's Theorem~\ref{Springer-theory-facts}, we obtain

  \begin{tabular}{lcl}
    $\gamma$ stabilizes the set $\big\{\zeta^{m_1},\ldots,\zeta^{m_n}\big\}$ & $\Leftrightarrow$ &~$c$ and  $c^k$ have the same eigenvalues\\
    & $\Leftrightarrow$ &~$c$ and  $c^k$ are conjugate\\
    & $\Leftrightarrow$ &~$c$ is $\zeta^k$-regular\\
    & $\Leftrightarrow$ & $\exists i\in \{1,\dots, n\}$, such that $\zeta^{-m_i}=\zeta^k$\\
    & $\Leftrightarrow$ & $k=-m_i$ for some exponent $m_i$ coprime to~$h$.
  \end{tabular}

\end{proof}

\begin{proof}[Proof of Proposition~\ref{prop:equiv}]
  \eqref{eq:freeaction}$\Leftrightarrow$\eqref{eq:degreeoffieldextension}:
    from Proposition~\ref{prop:refl-autom}, the action of~$\Gamma_W$ on $\Cox(W)$ is transitive. 
    So $|\Gamma_W|\geq |\Cox(W)|$, with equality if and only if the action is free. 
    We can conclude since $|\Gamma_W|=[K_W:\QQ]$, and from Lemma~\ref{lem:nccox}, $|\Cox(W)|=\varphi(h)/\varphi_W(h)$.
  
  \eqref{eq:invariantsubfield}$\Leftrightarrow$\eqref{eq:generatedbycoefficients}:
    follows trivially from Lemma~\ref{lem:zeta}.

  \eqref{eq:invariantsubfield}$\Leftrightarrow$\eqref{eq:degreeoffieldextension}:
    it follows from Lemma~\ref{lem:zeta} that $\QQ(\zeta)^{G_W}\subseteq K_W$.
    Therefore
    \[ 
      \QQ(\zeta)^{G_W} = K_W \quad \Leftrightarrow \quad [\QQ(\zeta)^{G_W}:\QQ] = [K_W:\QQ].
    \]
    Now we can compute
    \[
      [\QQ(\zeta)^{G_W}:\QQ] = \frac{[\QQ(\zeta):\QQ]}{[\QQ(\zeta):\QQ(\zeta)^{G_W}]}=\frac{\varphi(h)}{|G_W|}.
    \]
    So $K_W=\QQ(\zeta)^{G_W}$ if and only if  $[K_W:\QQ]=\varphi(h)/|G_W|$.
    The equivalence then follows since one has $|G_W|=\varphi_W(h)$ by Lemma~\ref{lem:exponents}.
\end{proof}

We can now deduce Theorem~\ref{thm:galois} by proving that any of the statements in Proposition~\ref{prop:equiv} holds.
Checking case-by-case the equality $[K_W:\QQ]=\frac{\varphi(h)}{\varphi_W(h)}$ is easy, using the tables in~\cite{BMR1995} for the exceptional groups, and a short computation for the infinite series.
However, we do not have a uniform explanation.
After having done this work, we learnt that the equality $K_W=\QQ(\zeta)^G$ had already been given by G.~Malle in~\cite[Theorem~7.1]{malle}; his proof is also by inspecting the classification.

\begin{remark}
  \label{rk:bad2}
  As mentioned in Remark~\ref{rk:bad}, there are several badly-generated irreducible groups for which the highest degree~$d_n$ is still regular: these are the groups in the series $G(2d,2,2)$ (for $d\geq 2$), and the exceptional groups $G_7$, $G_{11}$, $G_{12}$, $G_{13}$, $G_{19}$, $G_{22}$ and $G_{31}$.
  In this case, one could define Coxeter elements as regular elements of order~$d_n$, as for well-generated groups.
  For these groups, the arguments of the proof of Theorem~\ref{thm:main} still hold. 
  So one has characterizations of Theorem~\ref{thm:main}\eqref{def1} through \eqref{def4} for these Coxeter elements, and there is also an action of the Galois group $\Gal(K_W/\QQ)$ on the set of conjugacy classes of Coxeter elements.
  All the arguments of Proposition~\ref{prop:equiv} also hold, so this action is simply transitive if and only if the equality $[K_W:\QQ]=\frac{\varphi(d_n)}{\varphi_W(d_n)}$ holds. 
  By Remark~\ref{rk:bad}, this is the case for all the groups mentioned above except $G_{13}$, so these groups satisfy all the statements of Theorem~\ref{thm:galois} as well.
\end{remark}

\section{Coxeter elements and generating sets}
\label{sec:coxeter}

In this section, we prove characterizations~\eqref{def6} and~\eqref{def5} in Theorem~\ref{thm:main}, which relate Coxeter elements to some well-behaved generating sets of~$W$.

\subsection{Background on Shephard groups and generalized Coxeter systems}
\label{ssec:shephard}

Among the well-generated groups, (complexified) real reflection groups are special in that they have a presentation using a \defn{Coxeter system}.
In Section~\ref{ssec:charac-cox}, we will relate the usual notion of Coxeter elements arising from a Coxeter presentation to our notion of Coxeter elements.
We will actually treat this in a more general setting, which has the advantage of applying to another class of groups admitting also nice presentations, namely Shephard groups.

A \defn{Shephard group} is a subgroup of $\GL(V)$ that is the symmetry group of a regular complex polytope. 
These groups have been introduced and classified by G.~C.~Shephard in~\cite{shephard}, we also refer to~\cite{Cox1991} for more detail.
H.~S.~M.~Coxeter showed in~\cite{Cox1967} that every Shephard group has a generalized Coxeter structure, as defined below.

\begin{definition}
  \label{def:gencox}
  A generalized Coxeter system~$(W,S)$ is a group~$W$ together with a subset~${S \subseteq W}$, such that~$W$ has a presentation with~$S$ as generating set and relations
  \[ \underbrace{sts\dots}_{m_{s,t}\text{ factors}}=\underbrace{tst\dots}_{m_{s,t}\text{ factors}} \text{ for }s,t \in S\text{ with } s \neq t, \quad \text{and } s^{p_s} =1 \text{ for }s \in S,\]
  for some integers $m_{s,t},p_s \geq 2$, where moreover $p_s = p_{t}$ whenever $m_{s,t}$ is odd.
\end{definition}

Note that in this definition, we do not allow the labels~$m_{s,t}$ and~$p_s$ to be infinite, unlike in the standard notion of Coxeter systems where $m_{s,t}=\infty$ is usually included.
However, this poses no restriction to our setting since we will only be interested in the case where such a generalized Coxeter system gives rise to a finite group.

\medskip

Similarly to the situation for Coxeter systems, one can construct a Coxeter graph for a generalized Coxeter system $(W,S)$.
The vertices are given by the generators and are labelled by their order, i.e., the vertex~$s$ is labelled by~$p_s$.
Moreover, whenever~$m_{s,t} \geq 3$, the two vertices~$s$ and~$t$ are joined by an edge labelled by~$m_{s,t}$.
This yields the definition of irreducible generalized Coxeter systems as those for which the Coxeter graph is connected.

\medskip

Work of H.~S.~M.~Coxeter~\cite{Cox1967} and of D.~W.~Koster~\cite{Kos1975} provide the relations between the combinatorics of finite generalized Coxeter system and the geometry of real reflection groups and Shephard groups.

\begin{itemize}
  \item Any Shephard group has a structure of generalized Coxeter system (this is analogous to the well-known property for real reflection groups). 
  Moreover, as with the chamber geometry in the real case, there is a reasonably natural way to construct a set of generalized Coxeter generators consisting of reflections, using the geometry of the polytope defining the Shephard group~\cite{Cox1967}. 
  We call such a presentation constructed from the geometry a \defn{standard generalized Coxeter presentation}.

  \item Given a generalized Coxeter system $(W,S)$ and an $|S|$-dimensional complex vector space~$V$, there exist a representation $\rho : W \rightarrow \GL(V)$ and an Hermitian form on~$V$ which is invariant under~$\rho(W)$, such that for $s \in S$, the element $\rho(s)$ is a reflection of~$V$ of order~$p_s$.

  \item If the generalized Coxeter system is irreducible and finite, there is such a representation~$\rho$ which is faithful and such that~$\rho(W)$ is a Shephard group or a real reflection group.
\end{itemize}

The last point yields that finite, irreducible generalized Coxeter systems correspond to the class of complex reflection groups which is the union of (complexified) real irreducible reflection groups and Shephard groups.
The latter are known to be all real reflection groups with unbranched Coxeter graph, together with the infinite family $G(r,1,n)$ with $r\geq 3$, and~$15$ of the non-real irreducible, exceptional groups.

\subsection{Regular generating sets for irreducible well-generated groups}
\label{ssec:pres}

After we have recalled classical ways of constructing presentations for real reflection groups and for Shephard groups, we now discuss a less classical approach to construct well-behaved presentations of general irreducible well-generated groups.
We will then show that there exist such presentations whose generating set is regular, as needed in the assertion of Theorem~\ref{thm:main}.

\medskip

Any irreducible well-generated group of rank~$n$ can be minimally generated by~$n$ reflections.
Much work has been devoted to finding well-behaved presentations by generators and relations such that the generating set consists of~$n$ reflections.
For the general case of an arbitrary irreducible well-generated group, there is no canonical presentation analogous to the Coxeter presentation discussed above.
However, a uniform approach has been given by D.~Bessis~\cite{Bessis2}, using the geometry of the braid group of~$W$ and a construction known as the \emph{dual braid monoid}.
Fix an $e^{2i\pi/h}$-regular element~$c$, and denote by~$R_c$ the set of those reflections $r\in R$, such that $r\leq_R c$, where $\leq_R$ is the absolute order defined in~\eqref{eq:abs}. 
A \defn{dual braid relation} is a formal relation of the form $r_1r_2=r_3r_1$, with $r_1,r_2,r_3\in R_c$ such that $r_1r_2\leq_R c$, $\ell_R(r_1r_2)=2$, and the relation $r_1r_2=r_3r_1$ holds in~$W$.
Then~$W$ admits the presentation
\begin{equation}
  W \simeq  \left\langle\ R_c \ \middle| \ \text{dual braid relations} + \text{ relations } r^{p_r}=1 \ (\forall r\in R_c)\ \right\rangle , \label{eq:pres}
\end{equation}
where $p_r$ is the order of the reflection $r$ in~$W$.
(Removing the relations $r^{p_r}=1$ gives actually a presentation for the \emph{braid group}~$B(W)$, see \cite[Remark 8.9]{Bessis2}).
This presentation is highly redundant and can be simplified as follows.
For an expression $c = r_1 \cdots r_n$ with $r_1,\ldots,r_n \in R$, one can obtain a simplified presentation involving only the generators $r_1,\dots, r_n$,
\begin{equation}
  W \simeq  \left\langle\ r_1,\ldots,r_n \ \middle| \ \text{modified dual braid relations} + \text{ relations } r_i^{p_i}=1 \ (1 \leq i \leq n )\ \right\rangle. \label{eq:presred}
\end{equation}
This follows from the transitivity of the Hurwitz action on reduced decompositions of the Coxeter element~$c$, see~\cite[Definition~6.19, Proposition~7.6]{Bessis2}.
Indeed, using dual braid relations, any $r \in R_c$ can be written in terms of $r_1,\ldots,r_n$.
Thus, any occurrence of a generator $r \notin \{r_1,\ldots,r_n\}$ in a dual braid relation in~\eqref{eq:pres} is written in terms of $r_1,\ldots,r_n$ inside the modified dual braid relations in~\eqref{eq:presred}.
We call a presentation of~$W$ as in~\eqref{eq:presred} a \defn{dual braid presentation} for~$W$.

It is enough for our purposes to consider such dual braid presentations.
Nevertheless, such presentations usually still contain redundant relations, but removing these redundancies involves non-canonical choices.
Observe that, by construction, the product $r_1 \cdots r_n$ of the generators in this presentation is $e^{2i\pi/h}$-regular.

It turns out that the presentations of~$W$ obtained in earlier works can be obtained from a dual braid presentation by removing further redundancies.
In particular, the standard presentations of Coxeter groups and of Shephard groups can be obtained this way.
In the general case, one can also recover the presentations described by Coxeter in~\cite{Cox1967}, and the presentations given in~\cite{BMR98} (see also~\cite{Bessis-Michel} and~\cite[\S6]{Malle-Michel}).
Such well-behaved presentations are gathered in the table~\cite{Michel-table}, and are implemented explicitly in CHEVIE~\cite{CHEVIE}.

\medskip

We will next show that for any irreducible well-generated group~$W$, there exists a dual braid presentation that satisfies stronger properties.

\begin{proposition}
  \label{prop:standardpresreg}
  Any irreducible well-generated group~$W$ admits a dual braid presentation whose set of generators~$S_0$ satisfies the following properties:
  \begin{enumerate}[(1)]

    \item any reflection in~$W$ is conjugate to a power of a reflection in $S_0$; \label{it:powerconj}

    \item the product, in \emph{any} order, of all the elements in~$S_0$ is a Coxeter element, i.e., a regular element of order~$h$. \label{it:reordering}
  \end{enumerate}
  More precisely, the following presentations satisfy Properties~\eqref{it:powerconj} and~\eqref{it:reordering}:
  \begin{enumerate}[(i)]

    \item for~$W$ real: the standard Coxeter presentation (arising from the chamber geometry);

    \item for~$W$ Shephard group: the standard ``generalized Coxeter presentation'' for~$W$;

    \item for any~$W$: the explicit presentations in~\cite{Michel-table} implemented in CHEVIE.

  \end{enumerate}
\end{proposition}

Given a generating set consisting of $n$ reflections, Properties~\eqref{it:powerconj} and~\eqref{it:reordering} are sufficient to make the proof of Characterization~\eqref{def6} of Theorem~\ref{thm:main} work.
That is why we choose to give a name to such a generating set: we call below (and in Theorem~\ref{thm:main}) a \defn{regular generating set} for~$W$ a generating set which
\begin{itemize}
\item consists of $n$ reflections;
\item satisfies Properties~\eqref{it:powerconj} and~\eqref{it:reordering} in Proposition~\ref{prop:standardpresreg}.
\end{itemize}

\begin{remark}
  Given our definition of a dual braid presentation in~\eqref{eq:presred}, Proposition~\ref{prop:standardpresreg} implies that for an $e^{2i\pi/h}$-regular element~$c$, there exists a reduced decomposition $(r_1,\dots, r_n)$ of~$c$ such that the product, in \emph{any} order, of $r_1,\dots, r_n$ is a Coxeter element (i.e., regular of order~$h$). 
  This property does not depend on the choice of~$c$ as an $e^{2i\pi/h}$-regular element, or even as a Coxeter element (this follows easily from Proposition~\ref{prop:refl-autom}). 
  Note however that it \emph{does} depend in general on the chosen reduced decomposition.
  There may exist reduced decompositions $(r_1',\dots, r_n')$ of~$c$ such that for some $\sigma\in S_n$, the product $r'_{\sigma(1)}\dots r'_{\sigma(n)}$ is not a Coxeter element. 
  For example, in type $D_4$, consider the reduced decomposition $c=s\cdot t\cdot uvu \cdot u$ (where $s,t,u,v$ are the standard Coxeter generators, $u$ being the central one in the diagram). 
  Then the product $s\cdot uvu\cdot t \cdot u$ has order $4\neq 6=h$ and is thus not a Coxeter element.
\end{remark}

\begin{proof}[Proof of Proposition~\ref{prop:standardpresreg}.]
  We first show uniformly that Property~\eqref{it:powerconj} actually holds for any dual braid presentation.
  This property is classical for the standard presentations of real groups (see e.g.~\cite[Proposition~1.14]{Humphreys}) and of Shephard groups (a case-by-case argument already appears in~\cite[Theorem~5]{Kos1975}). 

  To check this property in general, let~$s$ be a reflection in~$W$, denote by~$H$ its fixed hyperplane, by~$W_H$ the pointwise stabilizer of~$H$ in~$W$, and by~$e_H$ the order of~$W_H$. 
  Recall that~$s$ is called a distinguished reflection if~$s$ has determinant $e^{2i\pi/e_H}$, see~\cite[Definition~2.15]{BMR98}.
  Bessis' constructions in~\cite{Bessis2} imply that any reflection in a generating set of any dual braid presentation, i.e., any reflection in~$R_c$ in~\eqref{eq:pres}, is a distinguished reflection (these reflections are indeed constructed from elements of the braid group called ``braid reflections'', see~\cite[Remark~6.10]{Bessis2}).
  As a consequence, given a dual braid presentation for~$W$, with generating set~$S$, Property~\eqref{it:powerconj} is equivalent to the fact that any reflecting hyperplane is in the~$W$-orbit of the hyperplane of some reflection in~$S$. 
  Since $S$ generates~$W$, this property follows from the description of linear characters of~$W$ using the orbits of hyperplane, cf. e.g. \cite[Theorem~9.19]{LehrerTaylor}.

  It remains to show the existence of a dual braid presentation satisfying Property~\eqref{it:reordering}. 
  We need to prove that for $\zeta = e^{2i\pi/h}$ and a~$\zeta$-regular element~$c$, there exists a reduced decomposition $(r_1,\dots, r_n)$ of~$c$ as in \eqref{eq:presred}, such that for any $\sigma\in S_n$, the product $r_{\sigma(1)}\dots r_{\sigma(n)}$ is a Coxeter element. 
  For the standard Coxeter presentation of a real reflection group, it is again a classical fact: since the associated Coxeter diagram is a tree, a standard argument yields that all these different products are conjugate, and so they are actually all~$\zeta$-regular.
  The same argument still holds for any Shephard group (because its generalized Coxeter diagram is linear), in particular for the series $G(d,1,n)$. 
  For the remaining groups, we need to resort to a case-by-case check.
  We actually show the stronger result that the different products are either~$\zeta$-regular or $\zeta^{-1}$-regular.
  \begin{itemize}
  \item For any of the exceptional well-generated irreducible groups, we check that the product of the minimal generating set implemented in CHEVIE, in any order, is either~$\zeta$-regular or $\zeta^{-1}$-regular, and thus satisfies Property~\eqref{it:reordering}.
  \item For the remaining family $W=G(e,e,n)$, we need to do the computation explicitly.
    Write $b_1,\dots, b_n$ for the canonical basis of $\CC^n$, and set $\omega:=\exp(2i\pi/e)$ and $\zeta:=\exp(2i\pi/h)$, where $h=(n-1)e$.
    Let $t,s_1,..., s_{n-1}$ be moreover the standard generators for~$W$. 
    The generators $s_1,..., s_{n-1}$ are the standard generators of type $A_{n-1}$: $s_i$ interchanges $b_i$ and $b_{i+1}$ and we write $s_i=(i\ i+1)$. 
    The generator $t$ is defined by $t(b_1)=\omega b_2$, $t(b_2)=\omega^{-1}b_1$, and $t(b_i)= b_i$ for $i\geq 3$.
    Let $c_0:=s_1...s_{n-1}=(123 \dots n)$ and $c:=tc_0$. 
    Computing the eigenvalues of~$c$,  we see that~$c$ is~$\zeta$-regular.
    Moreover, a product of the generators $t, s_1,\dots, s_{n-1}$, in any order, is conjugate to an element of the form $tw$ where $w=s_{\pi(1)} s_{\pi(2)}\dots s_{\pi(n-1)}$, for $\pi\in S_{n-1}$.
    The element $w$ is again a long-cycle of the permutation group $S_n$, so is conjugate to $c_0$. 
    Moreover, $w$ either sends $1$ to $2$ or sends $2$ to $1$ (depending on the relative position of the factors $s_1$ and $s_2$ in the product).
    Thus $tw$ is conjugate
    \begin{itemize}
    \item either to $t\cdot  (123 \dots n)=c$, and thus is~$\zeta$-regular;
    \item or to $t\cdot (21n\ n-1 \dots 3)=tc_0^{-1}$, which is conjugate to $c^{-1}$, and thus is $\zeta^{-1}$-regular.
    \end{itemize}
  \end{itemize}
\end{proof}

\begin{remark}
  \label{rk:altpres}
  Alternative presentations, other than the ones in~\cite{Michel-table}, have been given for the five exceptional non-Shephard well-generated groups; see in particular~\cite[\S6]{Malle-Michel}, where the new generating sets are obtained from an initial one by applying Hurwitz action.
  It is natural to ask whether these alternative presentations are also regular, i.e., whether the product of the generators, in any order, is a regular element of order $h$.
  It turns out that for $G_{24}$, $G_{27}$ and $G_{29}$, all the alternative presentations are also regular, whereas for $G_{33}$ and $G_{34}$, among the five presentations $P_1$ through $P_6$ described by G.~Malle and J.~Michel in~\cite[\S6.4]{Malle-Michel}, only the initial presentation $P_1$ is regular.
\end{remark}

We now prove the following equivalence, which corresponds to characterization~\eqref{def6} in Theorem~\ref{thm:main}.
Recall that two generating sets for~$W$ are said to be \emph{isomorphic} if there is a bijection between them which extends to an automorphism of~$W$.
\begin{proposition}
  \label{prop:pres}
  Let~$W$ be an irreducible, well-generated group, and fix any regular generating set~$S_0$ for~$W$.
  Let~$c$ be an element in~$W$.
  Then~$c$ is a Coxeter element of~$W$ if and only if there exists a subset~$S$ of reflections of~$W$ such that:
\begin{itemize}
  \item~$S$ is isomorphic to~$S_0$;
  \item~$c$ is the product (in some order) of the elements of~$S$. 
\end{itemize}
\end{proposition}

One direction of the equivalence is a consequence of the following property.

\begin{lemma}
  \label{lem:isom-reg}
  Let $S_0$ be a regular generating set for~$W$. Suppose $S$ is a generating set for~$W$, isomorphic to $S_0$ and consisting of reflections. Then~$S$ is also a regular generating set.
\end{lemma}

\begin{remark}
  Together with this fact, Proposition~\ref{prop:pres} implies that~$c$ is a Coxeter element if and only if it is a product of the elements in a regular generating set.
  However, there may be several isomorphism classes of regular generating sets, so Proposition~\ref{prop:pres} is stronger than this property.
  Note that even for real groups, there may be regular generating sets which are not isomorphic to the Coxeter generating set (consider for example the set of transpositions~$\{(12),(13),(14)\}$ in the symmetric group $S_4$).
\end{remark}

\begin{proof}[Proof of Lemma~\ref{lem:isom-reg}]
    Since $S$ is isomorphic to $S_0$, there exists an automorphism $\psi$ of~$W$ such that $S=\psi(S_0)$.
    Any reflection of~$W$ is a conjugate of a power of some reflection in $S_0$ (Property~\eqref{it:powerconj} for a regular generating set), and $S\subseteq R$, so $\psi$ sends reflections to reflections, and~$S$ satisfies Property~\eqref{it:powerconj} as well.
    Any product of elements of $S$ is the image by~$\psi$ of a product of elements of~$S_0$, which is a Coxeter element by Property~\eqref{it:reordering}.
    Since~$\psi$ is a reflection automorphism, by Proposition~\ref{prop:refl-autom} such a product is also a Coxeter element and ~$S$ satisfies Property~\eqref{it:reordering} as well.
\end{proof}

\begin{proof}[Proof of Proposition~\ref{prop:pres}]
  Let~$c$ be any Coxeter element of~$W$.  Let $c_0$ be a Coxeter element obtained by taking the product of the elements in $S_0$ in some order. From Proposition~\ref{prop:refl-autom}, there exists a reflection automorphism~$\psi$ such that $c=\psi(c_0)$. 
  Set~$S=\psi(S_0)$. 
  Then~$c$ is the product in some order of the elements in~$S$, and~$S$ is a generating set isomorphic to~$S_0$.

  The remaining implication follows from Lemma~\ref{lem:isom-reg}.
\end{proof}

\subsection{Characterization of Coxeter elements in real groups and Shephard groups}

\label{ssec:charac-cox}

Let~$W$ be an irreducible real reflection group or a Shephard group, with set of reflections~$R$. 
We are now in the position to check the characterization~\eqref{def5} of Coxeter elements in Theorem~\ref{thm:main}, i.e., that regular elements of order~$h$ in~$W$ are precisely those elements that can be written as $s_1\dots s_n$ for $S=\{s_1,\dots,s_n\}$, with $S\subseteq R$ and $(W,S)$ being a generalized Coxeter system (see Section~\ref{ssec:shephard}).
This will follow easily from the characterization~\eqref{def6} proven in Section~\ref{ssec:pres}, using Proposition~\ref{prop:cox} below that exhibits a rigidity property of generalized Coxeter presentations.

\begin{proposition}
\label{prop:cox}
  Let~$W$ be a complex reflection group, and let~$R$ be the set of all its reflections. Assume~$S$, $S'$ are two subsets of $R$ such that $(W,S)$ and $(W,S')$ are generalized Coxeter systems.
  Then $(W,S)$ and $(W,S')$ are isomorphic generalized Coxeter systems.
\end{proposition}

In particular, rewritten in the context of abstract Coxeter systems, Proposition~\ref{prop:cox} implies the following.
Let $(W,S)$ be a \emph{finite} Coxeter system, and let~$T$ denote the conjugacy closure of~$S$ in~$W$.
Then any Coxeter system $(W,S')$ for~$W$, with $S'\subseteq T$, is isomorphic to $(W,S)$.
Even in the case of classical Coxeter systems, the only proof of Proposition~\ref{prop:cox} we know is case-by-case. 
This property has been already stated in \cite[\S1.1]{Bessis1}, without proof.
Note that however, $S'$ is not necessarily $W$-conjugate to $S$, as seen in the example of the dihedral group $I_2(5)$ (Example~\ref{ex:H2}).

\begin{remark}
  It is well known that Proposition~\ref{prop:cox} does not hold if one does not assume~$S\subseteq R$.
  Some classical counterexamples arise from the existence of a group isomorphism between~$I_2(2m)$ and~${A_1\times I_2(m)}$ (for any~$m\geq 3$).
  This property was moreover shown not to hold in general, even with~$S \subseteq R$, for irreducible \emph{infinite} Coxeter groups, see~\cite{Mue2000}.
\end{remark}

\begin{lemma}
\label{lem:abstractisomorphism}
  Let~$(W,S)$ and $(W',S')$ be two irreducible generalized Coxeter systems such that~$W$ and~$W'$ are finite and isomorphic as abstract groups.
  Then $(W,S)$ and $(W',S')$ are isomorphic generalized Coxeter systems, i.e., they have isomorphic Coxeter graphs.
\end{lemma}

\begin{proof}
  It is well known that this property holds for finite irreducible Coxeter systems, see for example~\cite[Appendix~A1]{BB2005}.
  Using the classification given in~\cite{Kos1975}, it can as well be checked for finite irreducible generalized Coxeter groups\footnote{J.~Michel informed us that the following more general statement can be deduced from considerations in~\cite{MM2010}: except for straightforward coincidences, any two irreducible complex reflection groups of different Shephard-Todd types are non-isomorphic as abstract groups.}.
\end{proof}

\begin{proof}[Proof of Proposition~\ref{prop:cox}]
  Let  $W\subseteq \GL(V)$ be a reflection group which admits a generalized Coxeter system~$(W,S)$ such that~$S$ consists of reflections.
  
  First suppose that $(W,S)$ is not irreducible.
  The generating set~$S$ has a nontrivial partition $S_1\sqcup S_2$, where the reflections in $S_1$ commute with the ones in $S_2$.
  The group~$W$ then preserves the intersection of the hyperplanes associated to the reflections in~$S_1$ (which is a proper subspace) and thus,~$W$ acts reducibly on~$V$.
  
  Conversely, suppose now that~$W$ is reducible as a representation, i.e., $W\simeq W_1\times W_2$, $V=V_1\oplus V_2$, with $W_i\hookrightarrow \GL(V_i)$. 
  Then any element $s\in S$ preserves $V_1$ and $V_2$; but since~$s$ is a reflection, it then must act trivially on $V_1$ and as a reflection on $V_2$, or vice versa. 
  This yields a nontrivial partition $S=S_1\sqcup S_2$, with $\left< S_i \right> = W_i$, so that $(W,S)$ is a reducible generalized Coxeter system.
  
  In total, we obtain that the irreducible components of the representation~$W$ correspond to the irreducible components of the Coxeter system $(W,S)$. 
  We can therefore restrict the situation to the case where~$W$ is an irreducible reflection group, with two irreducible generalized Coxeter structures $(W,S)$ and $(W,S')$.
  This case was treated in Lemma~\ref{lem:abstractisomorphism}.
\end{proof}

We can now complete the proof of Theorem~\ref{thm:main}, by showing that the characterization~\eqref{def5} is equivalent to the others.

\begin{proof}[Proof of Theorem~\ref{thm:main}\eqref{def5}]
  In this proof we fix a standard generalized Coxeter system~$(W,S_0)$ for~$W$.
  From Proposition~\ref{prop:standardpresreg},~$S_0$ is a regular generating set.
  We write~$S=\{s_1,\dots, s_n\}$.
  From Proposition~\ref{prop:pres}, we can use characterization~\eqref{def6} of Theorem~\ref{thm:main}, which actually does not depend on the choice of the regular generating set.

  We first prove the implication~\eqref{def1}$\Rightarrow$\eqref{def5}.
  Suppose~$c$ is a Coxeter element.
  We aim to show that there is a set~$S$ of generalized Coxeter generators for~$W$ such that~$c$ is the product of the generators in~$S$.
  Let~$w = s_1 \cdots s_n$. 
  Since~$S_0$ is a regular generating set,~$w$ is regular of order~$h$.
  By Proposition~\ref{prop:refl-autom}, there is a reflection automorphism~$\psi$ of~$W$ sending~$w$ to~$c$.
  Next, set~$S = \psi(S_0) = \{ \psi(s_1),\ldots,\psi(s_n) \}$, so that~$c$ is the product of the elements in~$S$. 
  Since~$\psi$ is an automorphism preserving reflections, $(W,S)$ is again a generalized Coxeter system.

  To conclude, we prove \eqref{def5}$\Rightarrow$\eqref{def6}. 
  Let $S=\{t_1,\dots, t_n\}\subseteq R$ be such that $(W,S)$ is a generalized Coxeter system, and   $c:=t_1\cdots t_n$.
  It follows from Proposition~\ref{prop:cox} that $(W,S)$ and $(W,S_0)$ are isomorphic Coxeter systems. 
  So the generating set~$S$ is isomorphic to $S_0$, and~$c$ satisfies characterization~\eqref{def6}.
\end{proof}

\section{Regular elements and reflection automorphisms}
\label{sec:reg}

In this section we prove Theorem~\ref{thm:refl-autom-reg} on the action of reflection automorphisms on regular elements of a given fixed order. 
Statements~\eqref{it:trans-reg}-\eqref{it:divides} are all direct consequences of the first statement~\eqref{it:trans-reg} which is recalled in the following proposition.
\begin{proposition}
  \label{prop:refl-autom-reg}
  Let~$W$ be an irreducible complex reflection group, and~$d$ be a regular number for~$W$. 
  The action of~$\Aut_R(W)$ on~$W$ preserves the set of regular elements of order~$d$, and is transitive on it.
\end{proposition}

Most of the proof follows the same lines as for Coxeter elements in Section~\ref{sec:refl-autom}, simply by replacing~$h$ by~$d$.
However, it is not true in general that an element of order~$d$ is regular if and only if it has an eigenvalue of order~$d$ (characterization~\eqref{def2} in Theorem~\ref{thm:main} for Coxeter elements), so we need to check the existence of a regular eigenvector.

\subsection{Images of regular elements}
\label{ssec:image-reg}

We first prove that the image of a regular element by a reflection automorphism is regular.
We assume again that~$W$ is embedded in~$\GL_n(K_W)$.
Let~${w\in W}$ be a regular element of order~$d$, and~$\psi$ be a reflection automorphism of~$W$.
From the discussion in Section~\ref{sec:intro} and the use of Proposition~\ref{prop:marinmichel}, we know that~$\psi$ is of the form
\[
\begin{array}{llcl}
\psi: & W & \to & W\\
& x & \mapsto & a \bar{\gamma}(x)a^{-1}
\end{array}
\]
where~$\gamma$ is a field automorphism in~$\Gamma_W=\Gal(K_W/\QQ)$, and with~$\bar{\gamma}$ the associated automorphism of~$\GL_n(K_W)$, and some~$a \in \GL_n(\CC)$.

Let~$\zeta$ be a root of unity of order~$d$ such that~$w$ is~$\zeta$-regular, and let~$v$ be a regular eigenvector of~$w$ for~$\zeta$.
We can (and do) assume that the coordinates of~$v$ are in~$K:=K_W(\zeta)$.
We are going to construct from~$v$ a regular eigenvector for~$\psi(w)$.

Using Corollary~\ref{Galois-corollary}\eqref{it:extension1}, we extend~$\gamma$ to a field automorphism~$\delta$ of~$K$.
Note that~$\delta(\zeta)$ has the same minimal polynomial as~$\zeta$, so~$\delta(\zeta)=\zeta^p$ for some~$p$ coprime to~$d$.
Applying~$\delta$ to the coordinates of~$v$, we get a new vector in~$K^n$ that we denote by~$\bar{\delta}(v)$.
Applying the Galois action on the equation~$w(v)=\zeta v$, we get the equality~$\bar{\gamma}(w)(\bar{\delta}(v))=\zeta^p \bar{\delta}(v)$.
Then, denoting the vector~$a\bar{\delta}(v)$ by~$v'$, we obtain
\[
  \psi(w)(v')=a \bar{\gamma}(w)a^{-1} (a\bar{\delta}(v))= \zeta^p v'.
\]
To conclude that~$\psi(w)$ is regular of order~$d$, it remains to check that~$v'$ is a regular vector.
Note that~$v'$ is regular if and only if
\[ \forall r\in R, \ r(v')\neq v'.\]
Applying the Galois action and the action of~$a$ on an equation of the form~$r(u)=u$ (for some vector~$u\in K^n$ and some reflection~$r\in R$) gives us the equivalences
\[
    r(u)=u \quad \Leftrightarrow \quad \bar{\gamma}(r)(\bar{\delta}(u))=\bar{\delta}(u)
           \quad \Leftrightarrow \quad \psi(r)(a\bar{\delta}(u))=a\bar{\delta}(u).
\]
Thus, since~$\psi(R)=R$, the regularity of~$v$ implies the regularity of~$v'$.

\subsection{Transitivity of reflection automorphisms}
\label{ssec:trans-reg}

To conclude the proof of Proposition~\ref{prop:refl-autom-reg}, we need to prove that the action of~$\Aut_R(W)$ on the set of regular elements of order~$d$ is transitive.
Let~$w \in W$ be~$\zeta$-regular with~$\zeta$ being a primitive~$d$-th root of unity, and let~$w'$ be a~$\zeta^p$-regular element with~$p$ coprime to~$d$.
The construction of a reflection automorphism mapping~$w$ to~$w'$ is the same as in the case~$d=h$ in Section~\ref{ssec:trans}.
We start by extending the field automorphism defined by~$\zeta\mapsto \zeta^p$ to an automorphism~$\delta$ of the field~$K_W(\zeta)$, using Corollary~\ref{Galois-corollary}\eqref{it:extension2}.
We obtain a group~$\bar{\delta}(W)$, that is (by Corollary~\ref{cor:galoisconj}) conjugate to~$W$ in~$\GL_n(\CC)$, say~$\bar{\delta}(W)=gWg^{-1}$.
We can then define a reflection automorphism~$\psi$ of~$W$ by~$x \mapsto \psi(x)=g^{-1}\bar{\delta}(x)g$. 
Now, the same arguments as in Section~\ref{ssec:image-reg} show that~$\psi(w)$ is regular for the eigenvalue~$\delta(\zeta)=\zeta^p$.
By Springer's Theorem~\ref{Springer-theory-facts}\eqref{eq:springer1},~$\psi(w)$ is then conjugate to~$w'$ in~$W$, say~$w' = a^{-1}\psi(w)a$ for some~$a$ in~$W$.
Then the map~$x \mapsto  a^{-1}\psi(x)a$ yields a reflection automorphism of~$W$ mapping~$w$ to~$w'$.

\subsection{Galois action on conjugacy classes of regular elements}

The proof of the remaining statements in Theorem~\ref{thm:refl-autom-reg} is exactly the same as for Coxeter elements. 
Recall that we defined in Section~\ref{sec:intro} (before Proposition~\ref{prop:marinmichel}) a natural action of~$\Gamma_W=\Gal(K_W/\QQ)$ on the set of~$W$-conjugacy classes of regular elements. 
This yields an action of~$\Gamma_W$ on the set~$\Cox_d(W)$ of regular elements of order~$d$. 
Since any reflection automorphism is a Galois automorphism (Proposition~\ref{prop:marinmichel}), we get from Proposition~\ref{prop:refl-autom-reg} that this action is transitive, proving statement~\eqref{it:galois-trans} of Theorem~\ref{thm:refl-autom-reg}.

\medskip

The equality~$|\Cox_d(W)| = \varphi(d)/\varphi_W(d)$ holds since the same proof as for Lemma~\ref{lem:nccox} applies, by simply replacing~$h$ with~$d$.

The last statement~\eqref{it:divides} then follows directly, given that one has a transitive action of a finite group on a finite set.


\section*{Acknowledgements}
The authors thank Jean Michel for many helpful discussions during the preparation of this work, as well as Alex Miller, both for helpful conversations and for pointing them to the results in Koster's thesis~\cite{Kos1975}.
They also thank Vincent Beck, David Bessis, Christian Krattenthaler and Ivan Marin for enlightening discussions on this subject.
Finally, they are grateful to Gunter Malle for useful comments on a previous version of this article.


\bibliographystyle{alpha}
\bibliography{CoxeterElementNote.bib}

\end{document}